\documentclass[a4paper,11pt]{article}

\usepackage{makeidx}
\usepackage{amscd}
\usepackage{amsmath}
\usepackage{amssymb}
\usepackage{amsthm}

\usepackage[all]{xy}

\providecommand{\MR}{\relax\ifhmode\unskip\space\fi MR }

\providecommand{\href}[2]{#2}
\theoremstyle{plain}
\newtheorem{thm}{Theorem}[section]

\newtheorem{lem}[thm]{Lemma}
\newtheorem{cla}[thm]{Claim}

\newtheorem{prop}[thm]{Proposition}

\theoremstyle{definition}
\newtheorem{defn}[thm]{Definition}
\newtheorem{rem}[thm]{Remark}

\newcommand{\PP}{\mathbb P}
\newcommand{\Z}{\mathbb Z}

\newcommand{\C}{\mathbb C}

\newcommand{\ord}{\ensuremath{\operatorname{ord}}}
\newcommand{\adj}{\ensuremath{\operatorname{adj}}}

\newcommand{\Pic}{\mathop{\mathrm{Pic}}\nolimits}
\newcommand{\Ext}{\mathop{\mathrm{Ext}}\nolimits}

\newcommand{\Ker}{\mathop{\mathrm{Ker}}\nolimits}
\newcommand{\Cok}{\mathop{\mathrm{Coker}}\nolimits}

\newcommand{\Sym}{\operatorname{Sym}}

\newcommand{\str}{\mathcal{O}} 
\newcommand{\mc}{\mathcal}
\newcommand{\mb}{\mathbb}

\newcommand{\Supp}{\ensuremath{\operatorname{Supp}}}

\newcommand{\FM}{\ensuremath{\operatorname{FM}}}
\newcommand{\Fr}{\ensuremath{\operatorname{Fr}}}

\newcommand{\Span}[1]{\left<#1\right>}

\title{Elliptic ruled surfaces over arbitrary characteristic fields}

\author{Takato Togashi and Hokuto Uehara}
\date{}
\begin{document}
\maketitle
\begin{abstract}
Atiyah classifies vector bundles on elliptic curves $E$ over  
an algebraically closed field of any characteristic. 
On the other hand, a rank $2$ vector bundle on $E$ defines 
a surface $S$ with a $\PP ^1$-bundle structure on $E$. 
We study when $S$ has an elliptic fibration according 
to the Atiyah's classification, and what kinds of singular fibers appear.
\end{abstract}

\section{Introduction}
Kodaira initiated a study of elliptic surfaces over $\C$ about 60 years ago, since then,  we have a satisfactory theory of this subject. Bombieri and Mumford studied in \cite{MR491720} elliptic surfaces over an algebraically closed field $k$ of arbitrary characteristic $p$, 
and they introduced the notion of wild fibers in the case $p>0$, which turns out a main difficulty to develop an analogous theory of elliptic surfaces in the case $p>0$ to that in the case $p=0$. Ueno and Katsura studied in \cite{MR799664} what extent the theory in the case $p=0$ can be extended or have a nice analogy in the case $p>0$. 
In this article, we study elliptic fibrations on elliptic ruled surfaces in arbitrary $p$ with the aid of \cite{MR799664} and Atiyah's study of vector bundles on elliptic curves. 

Let $S$ be a smooth projective surface defined over $k$. Suppose that $S$ admits a relatively minimal elliptic fibration $\pi \colon S \rightarrow C$, and 
$$
\pi^* (p_i) = m_i D_i \mbox{ ($m_i$ is the multiplicity, $p_i \in C,  i = 1, 2, \dots, \lambda $)}
$$
are the multiple fibers of $\pi$. 
We know that
 $\mb R^1\pi_*\str _S \simeq \mc L_{\pi} \oplus \mc T_{\pi}$, where $\mc L_{\pi}$ is an invertible sheaf and $\mc T_{\pi}$ a torsion sheaf, which is known to be $0$ in the case $p=0$. 
The multiple fiber $\pi^* (p_i)$ is said to be \emph{wild} for a point $p_i\in \Supp\mc T_{\pi}$,
and \emph{tame} for a point $p_i\not \in \Supp\mc T_{\pi}$. 
The canonical bundles on $S$ and $C$ are related by  \emph{the canonical bundle formula}, which states that there is an isomorphism
$$
\omega_S  \simeq \pi^* (\omega _C \otimes \mc L_{\pi}^{-1}) \otimes \str _S (\sum_{i=1}^\lambda a_i D_i)
$$
for some integer $a_i$ with $0\le a_i \le m_i-1$. If $a_i\ne m_i-1$, then $\pi^*p_i$ is known to be a wild fiber. 

If $\kappa(S)=-\infty$, $S$ is either a rational surface obtained by $9$-points blow up of $\PP^2$ or an \emph{elliptic ruled surface}, that is, a surface with a $\PP^1$-bundle over an elliptic curve.  
Harbourne and Lang classified multiple fibers on rational elliptic surfaces in \cite{MR936813}. 
It turns out that rational elliptic surfaces have reducible fibers, but no wild fibers. 
On the other hand, 
we can readily see that an elliptic fibration on an elliptic ruled surface has no reducible fibers, and moreover all 
multiple fibers are of type ${}_m\mathrm{I}_0$, namely the reductions of all multiple fibers are smooth elliptic 
curves (Lemma \ref{lem:fiber_on_ruled}). 
One of the aim of this article is to determine when a given elliptic ruled surface has an elliptic fibration $\pi$, 
and to determine how many wild and tame fibers appear in $\pi$ (Theorem \ref{thm:main1}).

Let $\mc E$ be a normalized rank $2$ vector bundle on an elliptic curve $E$ and 
$$
f \colon S = \mathbb{P} (\mathcal{E} ) \rightarrow E
$$
 be the $\PP ^1$-bundle on $E$, determined by $\mc{E}$. 
Let us set $e = -\deg \mathcal{E}$. Then 
we see that $e=0$ or $-1$ if $S$ has an elliptic fibration.
If $e=0$ and $\mc E$ is indecomposable, it is easy to see that the isomorphism class of such vector bundle is unique. If $e=0$ and $\mc{E}$ is decomposable,
we see that  $\mc E\cong \mc O_E\oplus \mc L$ for some $\mc L\in \Pic^ 0 E$.
On the other hand, if $e=-1$,  $\mc{E}$ is always indecomposable.  We also note that there are many isomorphism classes of such vector bundles $\mc{E}$ on $E$, but all vector bundles give the unique isomorphism class of $\PP(\mc{E})$ (see \cite[Theorem V.2.15]{MR0463157}).  

Theorem \ref{thm:main1} below is the first main result in this article. 
In the tables contained therein, the symbol ${}^*$ stands for a wild fiber. 
Moreover, as mentioned above, if a multiple fiber $mD$ is tame, then $a=m-1$, hence we omit the value of $a$ in the list.  For example, $(2, 0/2^*)$ in the case (ii-3) stands for one tame fiber  
of type ${}_2\mathrm{I}_0$ with $a_1=1$ and one wild fiber of type ${}_2\mathrm{I}_0$ with $a_2=0$. 


\begin{thm}\label{thm:main1}
Let $\mc E$ be a normalized rank $2$ vector bundle on an elliptic curve $E$,
and $S = \mathbb{P} (\mathcal{E} )$ the associated $\PP^1$-bundle 
over $E$ with $\mc E$. 
\begin{enumerate}
\item
For $e = 0$, we have the following:
\begin{center}
\begin{tabular}{|c|c|c|c|}
\hline
       & $\mathcal{E}$ & $(a_i/m_i)$ if $S$ has an elliptic fibration  & $p$ \\ \hline \hline
(i-1)& $\str_E \oplus \str_E$ &no multiple fibers & $p\ge 0$ \\ \hline
(i-2)& $\str_E \oplus \mathcal{L}$,  $\ord \mc L=m>1$ &$(m,m)$  & $p\ge 0$ \\ \hline
(i-3)& $\str_E \oplus \mathcal{L}$,  $\ord \mc L=\infty$ & no elliptic fibrations &$p\ge 0$\\ \hline
(i-4)& indecomposable & no elliptic fibrations & $p=0$ \\ \hline
(i-5)& indecomposable & $(p-2/{p}^*)$ & $p>0$ \\ \hline
\end{tabular} 
\end{center}

Here $\mc L$ is an element of $\Pic ^0E$.
%
\item
Suppose that $e=-1$. Then $S$ has an elliptic fibration. The list of singular fibers are the following: 
\begin{center}
\begin{tabular}{|c|c|c|c|}
\hline
       &$(a_i/m_i)$  &E                                   & $p$ \\ \hline \hline
(ii-1)&$(2,2,2)$     &                                      & $p\ne 2$ \\ \hline
(ii-2)&$(1/2^*)$   &  \mbox{supersingular}      & $p= 2$\\ \hline
(ii-3)&$(2, 0/2^*)$ & \mbox{ordinary}         & $p= 2$\\ \hline
\end{tabular} 
\end{center}
\end{enumerate}
\end{thm}

Maruyama also considered a condition when elliptic ruled surfaces have an elliptic fibration 
\cite[Theorem 4]{MR280493}, by terms of elementary transformations of ruled surfaces (see Remark \ref{rem:maruyama}). Suwa also considered a similar condition in \cite[Theorem 5]{MR242198} in the case $p=0$.
In the case of $p\ne 2$, the result in Theorem \ref{thm:main1} was obtained in the first author's master thesis \cite{To11}.

We also notice that  the elliptic fibration in the case (ii-2) have a \emph{wild fiber of strange type} (see Remark \ref{rem:strange}).

In the proof of Theorem \ref{thm:main1}, we can construct a \emph{resolution of singular fibers} on elliptic 
ruled surfaces. More precisely we have following:
\begin{thm}\label{thm:reduction}
Let $f\colon S\to E$ be a $\PP^1$-bundle over an elliptic curve $E$ such that $S$ also has an elliptic fibration $\pi\colon S\to \PP^1$.
Then there is a finite surjective morphism $\varphi\colon F\to E$ from an elliptic curve $F$, fitting into the following diagram:
\begin{equation*}\label{eqn:2,0,p}
\xymatrix{ 
F \ar[d]_\varphi\ar@{}[dr]|{\square}   &  F\times \PP^1 \ar[l]_{f'} \ar[r]^{\pi'} \ar[d]^q  & \PP ^1   \ar[d]^\psi  \\
E      &  S      \ar[l]_{f}           \ar[r]^{\pi} &  \PP ^1. 
}
\end{equation*}
Here, $f'$ and $\pi'$ are natural projections, the left square is the fiber product diagram and the right square is obtained by the  Stein factorization of $q\circ \pi$.   
\end{thm}
We observe in Theorem \ref{thm:reduction} that by taking a suitable finite cover of $S$, we obtain an elliptic fibration $\pi'$ having milder singular fibers (in Theorem \ref{thm:reduction}, $\pi'$ actually has no singular fibers).
The existence of such resolutions was already observed under some suitable situations.
See, for example, \cite[\S 6, \S 7]{MR799664} and \cite[Theorem A]{MR1753506}. 
It is notable that the notion of a \emph{wild fiber of strange type} was introduced in \cite{MR799664} as an obstruction of the existence of such resolutions in their constructions. However, we can actually find a resolution even in the case (ii-2), $\pi$ has a wild fiber of strange type. 

The construction of this article is as follows.
In \S \ref{sec:pre} we refer and prove several results on elliptic surfaces and vector bundles on elliptic curves. 
In \S \ref{sec:singular_fiber}, we narrow down the candidates of singular fibers of elliptic fibrations 
on elliptic ruled surfaces.
 We prove Theorem \ref{thm:main1} in \S \ref{sec:main1} and Theorem \ref{thm:reduction} in \S \ref{sec:reduction}.

In \cite{MR3652778}, we apply the result in \cite{To11}, which is the same as Theorem \ref{thm:main1} in the case $p\ne 2$, to study the set $\FM(S)$ of Fourier--Mukai partners of elliptic ruled surfaces $S$ in the case $p=0$. 
In the forthcoming paper \cite{UW}, we apply Theorem \ref{thm:main1} to study $\FM (S)$ for elliptic ruled surfaces $S$ for arbitrary $p$. 


\paragraph{Notation and conventions}\label{subsec:notation_convention}
All varieties $X$ are defined over an algebraically closed field $k$
of characteristic $p\ge 0$. A point $Q\in X$ always means a closed point. 
If we denote 
$$D_1\sim D_2\quad (\mbox{resp. }D_1 \equiv D_2)$$ 
for divisors $D_1$ and $D_2$ on a normal variety $X$,
we mean that $D_1$ and $D_2$ are linearly equivalent (resp. numerically equivalent).
We denote the dimension of the cohomology $H^i(X,\mc{F})$ of a sheaf $\mc{F}$ on $X$ by 
$h^i(X,\mc{F})$.
 
In the case $p>0$,
we denote  a relative Frobenius morphism by 
$$
\Fr\colon X_p\to X.
$$
 
By an \emph{elliptic surface}, we will
always mean a smooth projective surface $S$ together with a smooth projective curve $C$ and a relatively minimal projective morphism $\pi\colon S\to C$ whose general fiber is an elliptic curve.

Take a point $Q$ on an elliptic curve $E$. Since we have 
$$\Ext^1_E(\mc{O}_E,\mc{O}_E)\cong \Ext^1_E(\mc{O}_E(Q),\mc{O}_E)\cong k,$$ 
there is the unique isomorphism class of rank $2$ vector bundle
 $\mc{E}_{2,0}$ (resp. $\mc{E}_Q$) fitting into the non-split exact sequence:
$$
0\to \mc{O}_E\to \mc{E}_{2,0}\to \mc{O}_E\to 0 \quad(\mbox{resp. } 0\to \mc{O}_E\to \mc{E}_Q\to \mc{O}_E(Q)\to 0). 
$$

\paragraph{Acknowledgments.}
We would like to thank Hiroyuki Ito, Kentaro Mitsui, Shigeru Mukai,
Toshiyuki Katsura, Noboru Nakayama for invaluable suggestions. 
H.U. is supported by the Grants-in-Aid for Scientific Research (No. 18K03249). 


\section{Preliminaries}\label{sec:pre}

\subsection{Elliptic surfaces}\label{subsec:elliptic}
In this subsection, we refer several useful results on elliptic surfaces in \cite{MR491720} and \cite{MR799664}.
Let $\pi : S \rightarrow C$ be an elliptic surface. 
Suppose that 
$$\bigl \{ \pi^* (Q_i) = m_i D_i \bigm |  i = 1, 2, \dots, \lambda \bigr\}$$
is the set of multiple fibers for $Q_i\in C$. Then, we have a decomposition 
$\mb R^1\pi_*\str _S \simeq \mc L_{\pi} \oplus\mc T_{\pi}$ where $\mc L_{\pi}$ is an invertible 
sheaf and $\mc T_{\pi}$ is a 
torsion sheaf. It is known that $\mc T_{\pi}=0$ when $p=0$.
If $Q_i\notin \Supp\mc T_{\pi}$, $\pi^* (Q_i)$ is said to be a \emph{tame fiber},
 and a \emph{wild fiber} otherwise. Note that for a point $Q\in \Supp\mc T_{\pi}$, $\pi^*(Q)$ is a 
multiple fiber, since $h^0(\pi^*(Q),\mc O_{\pi^*(Q)})\ge 2$.
 
Let us define $d := \deg \mc L_{\pi}$ and $g:=g(C)$. 


\begin{prop}[Theorem 2 in \cite{MR491720}]\label{prop:canonical}
Let $\pi \colon S \longrightarrow C$ be a smooth projective elliptic surface. For the decomposition
$\mb R^1\pi_*\str _S \simeq \mc L_{\pi}\oplus \mc T_{\pi},$
we have
$$
\omega_S  \simeq \pi^* (\omega _C \otimes \mc L_{\pi}^{-1}) \otimes \str _S (\sum_{i=1}^\lambda a_i D_i).
$$
Moreover we have the following:
\begin{enumerate}
\item 
$0\leq a_i\leq m_i-1$ for  all $i = 1, 2, \dots, \lambda$.
\item 
If $m_i D_i$ is a tame fiber, then $a_i=m_i-1$.
\item 
$d = -\chi (\str _S) - h^0 (C,\mc T_{\pi})$ and $d\le 0$.  
\end{enumerate} 
\end{prop}
\noindent 
For each $i$, $\str _S (D_i)|_{D_i}$ is known to be a torsion element in $\Pic^0 D_i$.
Let us define $\nu _i:= \ord \str _S (D_i)|_{D_i}$. Then it is known that 
\begin{equation}\label{eqn:mpnu}
m_i=p^\alpha \nu_i
\end{equation}
 for some $\alpha\ge 0$ and that $\nu_i=m_i$ if and only if 
$m_iD_i$ is a tame fiber. 
If $D_i$ is a supersingular elliptic curve, $\Pic^0 D_i$ has no torsion elements whose order 
is divisible by $p$. Thus we have the following.

\begin{lem}\label{lem:multiple_sup}
Suppose that $mD$ is a multiple fiber and $D$ is a supersingular elliptic curve. Then
$mD$ is tame if and only if $m$ is not divisible by $p$. 
\end{lem}

The following is useful.

\begin{thm}[Corollary to Proposition 4 in \cite{MR491720}]\label{thm:MR491720}
Let $\pi \colon S \rightarrow \mathbb{P}^1$ be a smooth projective elliptic surface satisfying 
$h^1(S, \str _S) \leq 1$.
Then we have $a_i + 1 = m_i$ or $a_i + \nu _i + 1 = m_i$.
\end{thm}


\begin{defn}
In the above notation, assume furthermore  that $C= \mathbb{P}^1$ and $\chi (\str _S) = 0$. 
Such an elliptic surface $S$ is said to be of type $(m_1, \dots,m_\lambda | \nu _1, \dots, \nu_\lambda )$.
Assume furthermore that all singular fiber are tame. Then $S$ is said to be of type
 $(m_1, \dots, m_\lambda)$.
\end{defn}

The following is used as a necessary condition for algebraicity of elliptic surface in \cite{MR799664}.


\begin{thm}[Theorem 3.3 in \cite{MR799664}]\label{thm:algebraic}
Let $\pi : S \rightarrow \mathbb{P}^1$ be an
 elliptic surface with $\chi(\mc{O}_S)=0$ and suppose that $S$ is of type  
$(m_1, \dots, m_\lambda | \nu _1, \dots, \nu_\lambda )$.
Then for each $i\in \{  1, 2, \dots, \lambda\}$,
there are integers $n_1, n_2, \dots, n_\lambda$ satisfying
\begin{itemize}
\item
$n_i \equiv 1 \ (\bmod \ \nu _i)$, and
\item 
$n_1 /m_1 + n_2 /m_2 + \dots + n_\lambda /m_\lambda  \in \mathbb{Z}$.
\end{itemize}
\end{thm}

\begin{rem}\label{rem:m1m2}
Let $S$ be an elliptic surface satisfying $\chi(\str _S)=0$, and $C=\PP^1$.
In the following cases, the information on $m$ and $\nu$ are easily deduced 
from Theorem \ref{thm:algebraic}.
\begin{enumerate}
\item (\cite[Corollary 4.2]{MR799664})
Suppose that $S$ has a unique multiple fiber, and the type of $S$ is  $(m | \nu)$.
Then the multiple fiber is wild, $m = p^\alpha$ for some integer $\alpha>0$, and $\nu = 1$.  
\item (\cite[Corollary 4.3]{MR799664})
Suppose that $S$ is of type $(m_1,m_2)$.
Then  there are integers $n_2,n'_1$ satisfying 
$$
1/m_1 + n_2 /m_2,\quad  n'_1/m_1 + 1/m_2  \in \mathbb{Z}.
$$
These conditions imply the equality  $m_1=m_2$. 
\end{enumerate}
\end{rem}


\subsection{Atiyah's result}\label{subsec:atiyah}
Atiyah classified indecomposable vector bundles on elliptic curves \cite{MR131423}.
We summarize his results we need below.

Let $\mc M_E (r,d)=\mc M (r,d)$ be the set of isomorphism classes of indecomposable vector 
bundles of rank $r$ and degree $d$ over an elliptic curve $E$.
For $\mc E\in \mc M(r,d)$, every other vector bundle in  $\mc M(r,d)$
is of the form $\mc E\otimes \mc L$ for some $\mc L\in \Pic ^0 E=\mc M(1,0)$. 
There is the unique element $\mc E_{r,0}$  in $\mc M(r,0)$ such that
$h^0(\mc E_{r,0})\ne 0$. 
Furthermore for $\mc E\in \mc M(r,0)$, we have
\begin{align*}
&h^0(E,\mc E)=h^1(E,\mc E)=0 \mbox{ when $\mc E\ne \mc E_{r,0}$, and } \\
&h^0(E,\mc E_{r,0})=h^1(E,\mc E_{r,0})=1  
\end{align*}
Actually we can define $\mc E_{r,0}$ by 
putting $\mc E_{1,0}=\mc O_E$ and
the unique non-trivial extension 
\begin{equation}\label{eqn:Er0}
0\to \mc E_{r,0} \to \mc E_{r+1,0} \to \mc O_E \to 0
\end{equation}
inductively. We can also see that $\mc{E}_{r,0}\cong(\mc{E}_{r,0})^\vee$.

 Next, let us define the unique indecomposable vector bundle $\mc E_{P}$ of rank $2$ for a point $P\in E$,
which fits into the following non-split exact sequence
\begin{equation}\label{eqn:Er1}
0\to \mc O_E \to \mc E_{P} \to \mc O_E(P) \to 0.
\end{equation}

Although $\mc{E}_P\not\cong\mc{E}_Q$ for distinct points 
$P,Q\in E$, we have $\mathbb{P}(\mc{E}_P)\cong\mathbb{P}(\mc{E}_Q)$ 
\cite[Theorem V.2.15]{MR0463157}. We can also see that 
$$\mc{M}(2,1)=\{ \mc{E}_P \mid P\in E\}.$$

\subsection{Elliptic ruled surfaces}\label{subsec:elliptic_ruled}
We use the terminology on ruled surface in \cite[V.2]{MR0463157}, without specified otherwise.
Let $f \colon S \rightarrow E$ be a  $\PP^1$-bundle structure over an  elliptic curve $E$.
Denote a general fiber of $f$ by $F_f$. 
Then there is a rank $2$ vector bundle $\mc E$ on $E$ such that $S\cong \mathbb{P}(\mathcal{E})$ 
and we assume that $\mc E$ is normalized. 
Put $e=-\deg \mathcal{E}$ and take a minimal section $C_0$ with ${C_0}^2=-e$.
Let us define a divisor $D$ which satisfies $\mc{O}_S(D)\cong \det \mc{E}$.
Then we have 
$$K_{S} \sim -2C_0 -f^* D    \equiv -2C_0 -eF_f .$$

Note that if $S$ has an elliptic fibration $\pi\colon S\to \PP^1$, then $-K_S$ is nef.  
Furthermore, we can easily see that $-K_S$ is nef if and only if $e=0,-1$.
We can also deduce from \cite[V.2]{MR0463157} that 
\begin{equation}\label{eqn:possibility}
\mc{E}\cong
\begin{cases}
\mc{O}_E\oplus \mc{L} \ \mbox{ or }\ \mc{E}_{2,0} \quad &\mbox{in the case $e=0$}\\
\mc{E}_Q \quad &\mbox{in the case $e=-1$}
\end{cases}
\end{equation}
for some $\mc{L}\in \Pic^0 E$ and $Q\in E$.

Moreover, any rational curves on $S$ cannot dominate the elliptic curve $E$ by $f$.  
This fact immediately implies the following.

\begin{lem}\label{lem:fiber_on_ruled}
An elliptic ruled surface $S$ has no quasi-elliptic fibrations. Moreover, if $S$ has an elliptic fibration,
its singular fiber is of the of the type ${}_m\text{I}_0$ for $m>0$, that is, a multiple fiber whose reduction is a smooth elliptic curve. 
\end{lem}

\subsection{Non-existence of elliptic fibrations}
In this subsection, we show the non-existence of elliptic fibrations on certain elliptic ruled surfaces.
First we prove the following two lemmas. 

\begin{lem}\label{lem:p0}
For $m\ge 0$, we have
\begin{equation*}
h^0(E,\Sym ^m\mc E_{2,0})=
\begin{cases}
1 & \text{$p=0$ or $m<p$} \\
2 & \text{$m=p>0$}.
\end{cases}
\end{equation*}
\end{lem}

\begin{proof} 
There is a natural exact sequence
\begin{equation}\label{eqn:indecomposable}
0\to \Sym ^{r-1}\mc E_{2,0} \to  \Sym ^{r}\mc E_{2,0} \to \mc O_E \to 0,
\end{equation}
and Atiyah observed in \cite[Theorem 9]{MR131423}  that  in the case $p=0$ or $1\le r<p$,
$\Sym ^{r-1}\mc E_{2,0}\cong \mc E_{r,0}$, and we can see that (\ref{eqn:Er0}) 
is isomorphic to (\ref{eqn:indecomposable}).

In the case $r=p>0$, we can see that the exact sequence 
(\ref{eqn:indecomposable}) splits as follows.
Let us consider an affine open covering $\{U_i\}_i$ of $E$. 
Note that $\mc E_{2,0}$ is trivialized on each $U_i$.
Transition functions of $\mc E_{2,0}$ on $U_i\cap U_j$ can be describes as 
\begin{equation*}
\begin{pmatrix}
1&f_{ij} \\
0& 1
\end{pmatrix}
\end{equation*}
by (\ref{eqn:indecomposable}), where $f_{ij}$ is a regular function on $U_i\cap U_j$.
Then the
transition functions of $\Sym ^m\mc E_{2,0}$ is 
\begin{equation*}
\begin{pmatrix}
1&mf_{ij}& _mC_2 f_{ij}^2&_mC_3 f_{ij}^3     &\cdots & mf_{ij}^{m-1} &f_{ij}^m \\
0& 1     & (m-1)f_{ij}   &_{m-1}C_2 f_{ij}^2 &\cdots & (m-1)f_{ij}^{m-2} &f_{ij}^{m-1}\\
0& 0     & 1             & (m-2)f_{ij}       &\cdots & (m-2)f_{ij}^{m-3} &f_{ij}^{m-2}\\
\vdots &\vdots& \vdots   &\vdots   &\cdots                & \vdots & \vdots\\
0& 0     & 0             & 0       &\ldots           & 1 & f_{ij}\\
0& 0     & 0             & 0       &\ldots           & 0 &1
\end{pmatrix}.
\end{equation*}
We call this $(m+1) \times (m+1)$ matrix $A_{ij}^{(m)}$.
When $m=p$, it is the following form
$$
\arraycolsep5pt
\left(
\begin{array}{@{\,}c|cccc@{\,}}
1&0&\cdots&0&f_{ij}^{p}\\
\hline
0&&&&\\
\vdots&\multicolumn{4}{c}{\raisebox{-10pt}[0pt][0pt]{\Huge $A_{ij}^{(p-1)}$}}\\
0&&&&\\
\end{array}
\right),
$$
since $_pC_k$ is divisible by $p$.
When $E$ is an ordinary (resp. supersingular) elliptic curve,
the Frobenius morphism $\Fr$ induces a bijection (resp. a zero map)
$$
\Fr^*\colon H^1(E_p,\mc O_{E_p})\to H^1(E,\mc O_{E}).
$$
Thus we put $\{f_{ij}^p\}=\lambda \{f_{ij}\}\in H^1(E,\mc O_E)$
for some $\lambda \in k^*$ (resp. $\lambda =0$). In particular,
there are regular functions $g_i$ on $U_i$ such that 
$f_{ij}^p- \lambda f_{ij}=g_i-g_j$.

Take an $(p+1)\times (p+1)$ matrix $P_{i}$ as 
$$
\arraycolsep5pt
\left(
\begin{array}{@{\,}c|ccccc@{\,}}
1&0&\cdots&0&\lambda &-g_i\\
\hline
0&&&&\\
\vdots&\multicolumn{4}{c}{\raisebox{-10pt}[0pt][0pt]{\Huge $I_{p}$}}\\
0&&&&\\
\end{array}
\right).
$$
Here $I_{p}$ is the $p\times p$ identity matrix. 
Then we can see that 
$$
A_{ij}^{(p)}P_i=P_j\widetilde{A_{ij}}^{(p)},
$$
where we put
$$
\widetilde{A_{ij}}^{(p)}:=
\arraycolsep5pt
\left(
\begin{array}{@{\,}c|cccc@{\,}}
1&0&\cdots&0&0\\
\hline
0& &      &\\
\vdots&\multicolumn{4}{c}{\raisebox{-10pt}[0pt][0pt]{\Huge $A_{ij}^{(p-1)}$}}\\
0& &      &\\
\end{array}
\right).
$$
This means that the vector bundles defined by the transition functions $\{A_{ij}\}$
and $\{\widetilde{A_{ij}}\}$ are isomorphic to each other.
This completes the proof.
\end{proof}

\begin{lem}\label{lem:p1}
Take $\mc L\in \Pic ^0E$ and $m\in \Z_{\ge 0}$. 
Then we have
\begin{equation*}
h^0(E,\Sym^m( \str_E \oplus \mathcal{L} ))=
\begin{cases}
1 & m<\ord \mc L\\
2 & m=\ord \mc L.
\end{cases}
\end{equation*}
\end{lem}

\begin{proof}
The assertion follows from the isomorphism
\begin{align*}\label{ali:sym}
\Sym ^m (\str_E \oplus \mathcal{L}) 
\cong \str_E \oplus \mathcal{L} \oplus \cdots \oplus {\mathcal{L}}^{\otimes m}.
\end{align*}
\end{proof}

\begin{prop}\label{prop:non-existence}
Let $\mc{E}$ be a vector bundle of rank $2$ on an elliptic curve $E$. In the following cases, $S:=\PP_E(\mc{E})$ has 
no elliptic fibrations.
\begin{enumerate}
\item
$p=0$ and $\mc{E}=\mc{E}_{2,0}$. 
\item
$p\ge 0$ and $\mc{E}=\mc{O}_E\oplus \mc{L}$ for $\mc{L}\in \Pic ^0E$
with $\ord \mc{L}=\infty$. 
\end{enumerate}
\end{prop}

\begin{proof}
In both cases, we have $e=0$ and hence $-K_S\sim 2C_0+f^*D$, where $f\colon S\to E$ is the $\PP^1$-bundle structure and $D$ is a divisor satisfying $\mc{O}_S(D)\cong \det \mc{E}$.
If $S$ has an elliptic fibration, the complete linear system $|-mK_S|$ for some $m>0$ defines it, and hence
$h^0(S,\omega_S^{\otimes {-m}})\ge 2$. On the other hand, we have
$$h^0(S,\omega_S^{\otimes {-m}})=h^0(S,\mc{O}_S(2mC_0+mf^*D))= h^0(E,\Sym^{2m}(\mc{E})\otimes \mc{O}_S(mD)),$$
which equals to $1$ in the case (i) by Lemma \ref{lem:p0}, and in the case (ii) by Lemma \ref{lem:p1}.
Thus the assertion follows.
\end{proof}

\subsection{Existence of elliptic fibrations}
In this subsection, we show the existence of an elliptic fibration on a given elliptic ruled surface $S=\PP(\mc{E})$.
If $S$ has an elliptic fibration, then $-K_S$ is nef, hence $\mc{E}$ is one of vector bundles appeared in \eqref{eqn:possibility}.

To achieve our purpose, first we study the structure of $\varphi^*\mc{E}$, where
$\varphi\colon F\to E$ is a finite isogeny of elliptic curves.
For a given such $\varphi$,  there is a factorization 
\[
\xymatrix{\varphi\colon  F\ar[r]^\psi  &E' \ar[r]^\phi &E,
}
\]
where $\psi$ is purely inseparable and $\phi$ is separable.
It is known that $\psi$ is a composition of Frobenius morphisms.
If $\ker\phi$ is non-trivial, it is a finite abelian group, which contains a cyclic group of a prime order.
Thus $\phi$ is decomposed into isogenies with a prime degree.
Hence it is essential to study the structure of $\varphi^*\mc{E}$ for an isogeny
$\varphi$ with a prime degree.

Suppose that $p>0$.
Let $E$ be an ordinary elliptic curve.
Then, $E$ contains a subgroup $\Z/p\Z$ in a unique way, and 
let us call 
$$
q_E\colon E\to E/(\Z/p\Z)
$$
the quotient morphism. 
Recall that we have
$$\widehat{q_E}=\Fr\colon  E/(\Z/p\Z)\cong E_p\to E.$$ 
To the contrary
if $E$ is supersingular, we have 
$$
\widehat{\Fr}=\Fr\colon E_p\to E.
$$

For an isogeny  $\varphi\colon F\to E$ of degree $p$, 
note that 
\begin{equation*}
\varphi=\begin{cases}
\Fr \text{ if $\varphi$ is purely inseparable},\\
q_F \text{ if $\varphi$ is separable. (In this case, $E$ is necessarily ordinary.)}
\end{cases}
\end{equation*}

We use the following Lemma to prove Lemma \ref{lem:pull_back}.
\begin{lem}[Corollaries 1.7 and 2.6 in \cite{MR318151}]\label{lem:push_forward}
Let $\varphi \colon F\to E$ be an isogeny of elliptic curves of degree $n$.
Then we have 
\begin{equation*}
\varphi_*\mc{O}_F=
\begin{cases}
\mc{E}_{p,0} & \text{ if $n=p$ and  $\widehat{\varphi}$ is purely inseparable}\\
\bigoplus_{\varphi^*\mc{L}\cong \mc{O}_F}\mc{L} & \text{ if $\widehat{\varphi}$
is separable.}
\end{cases}
\end{equation*}
\end{lem}


\begin{lem}\label{lem:pull_back}
Let $F$ and $E$ be elliptic curves. Take an indecomposable vector bundle $\mc{E}_{Q}\in \mc{M}_E(2,1)$ for some $Q\in E$.
\begin{enumerate}
\item
Suppose that $p>0$, and let $\varphi\colon F\to E$ be an isogeny  of
degree $p$ with $\widehat{\varphi}$ purely inseparable. Then we have
$$\varphi^*\mc{E}_{2,0}\cong \mc{O}_{F}\oplus \mc{O}_{F}.$$
\item
Suppose that $p=2$. Then we have 
$$
\Fr^* \mc{E}_{Q}\cong \mc{E}_{2,0}\otimes \mc{O}_{E_2}(Q).
$$ 
\item
Suppose that $p\ge 0$.  Let $\varphi\colon F\to E$ be a separable isogeny of degree $2$.
Then  
$$\varphi^*\mc{E}_{Q}\cong \mc{O}_F(Q_1)\oplus \mc{O}_{F}(Q_2),$$
where $Q_i\in F$ and $\mc{O}_F(Q_1-Q_2)$ is an order $2$ element of 
$\Pic ^0F$.
\end{enumerate}
\end{lem}

\begin{proof}
Let $\varphi\colon F\to E$ be an isogeny with prime $n:=\deg \varphi$.
Note that  Lemma \ref{lem:push_forward} implies 
\begin{equation*}
\Cok \adj= 
\begin{cases}
\mc{E}_{p-1,0} & \text{ if $n=p$ and $\widehat{\varphi}$ is purely inseparable, }\\
\bigoplus_{\ord \mc{L}=n}\mc{L} & \text{ if $\widehat{\varphi}$ is separable,}
\end{cases}
\end{equation*}
where 
$$
\adj \colon \mc{O}_E\to \varphi_*\varphi^*\mc{O}_E=\varphi_*\mc{O}_F
$$
is the adjunction morphism. 

(i) By the assumption, the connection morphism 
$$
H^0(\Cok \adj)\to H^1(\mc{O}_E)
$$ 
is isomorphic, and hence the morphism
$$
\varphi^*=H^1(\adj)\colon \Ext_E^1(\mc{O}_E, \mc{O}_E)\cong H^1(\mc{O}_E) 
\to \Ext^1_{F}(\varphi ^*\mc{O}_E, \varphi^*\mc{O}_E)\cong H^1(\varphi_*\mc{O}_F)
$$ 
is zero.  (Note that  if $E$ is 
supersingular,  this fact immediately follows from the definition.)
Hence the short exact sequence obtained by $\varphi^*$ of the sequence   
\eqref{eqn:Er0} for $r=1$ splits. Then the result follows.

(ii), (iii) 
In both cases, we have 
$$
h^0(\Cok (\adj\otimes \mc{O}_E(-Q)))=0, 
$$
and hence the morphism 
$$H^1(\adj\otimes \mc{O}_E(-Q))\colon
H^1(\mc{O}_E(-Q))\to  H^1(\varphi_*\varphi^*\mc{O}_E\otimes \mc{O}_E(-Q))=H^1(\varphi^*\mc{O}_E(-Q))
$$
is injective.
Since it coincides with the morphism
$$
\varphi^*\colon \Ext^1_E(\mc{O}_E(Q),\mc{O}_E)\to \Ext^1_F(\varphi^*\mc{O}_E(Q),\varphi^*\mc{O}_E),
$$
the short exact sequence 
\begin{equation}\label{eqn:pullback_Q}
0\to \mc{O}_F\to \varphi^*\mc{E}_{Q}\to \varphi^* \mc{O}_E(Q)\to 0
\end{equation}
obtained by $\varphi^*$ of the sequence \eqref{eqn:Er1}
does not split. 

Assume that $p=2$, $\varphi=\Fr$ and $\Fr^*\mc{E}_{Q}$ is decomposable. Since $\Fr^*Q=2Q$, we have 
$$
\Fr^*\mc{E}_{Q}=\mc{O}_{E_2}(Q)\oplus \mc{O}_{E_2}(Q).
$$ 
Then it follows from \eqref{eqn:pullback_Q} that there is an exact sequence
$$
0\to \mc{O}_{E_2}(-Q)\to \mc{O}_{E_2}\oplus \mc{O}_{E_2}\to \mc{O}_{E_2}(Q)\to 0.
$$
But this is absurd by $h^0( \mc{O}_{E_2}(Q))=1$. Hence, $\Fr^*\mc{E}_{Q}$ is indecomposable of degree $2$ 
with $\det \Fr^*\mc{E}_{Q}\cong \mc{O}_{E_2}(2Q)$.
This completes the proof of (ii).

 Next, let us give the proof of (iii). Set $\varphi^*Q=Q_1+Q_2$ for distinct points 
 $Q_1$ and $Q_2$ on $F$. Then we see that $\det \varphi^*\mc{E}_{Q}\cong \mc{O}_F(Q_1+Q_2)$ by  
 \eqref{eqn:pullback_Q}.
Note that  $\mc{O}_F(Q_1-Q_2)$ is an order $2$ element of $\Pic ^0F$, since $\varphi$ is a separable isogeny 
 of degree $2$. 
Assume that $\varphi^*\mc{E}_{Q}$ is indecomposable. Then there is an element $\mc{L}\in \Pic ^0F$ such that
$\varphi^*\mc{E}_{Q}\otimes \mc{O}_{F}(-Q_2)\cong \mc{E}_{2,0}\otimes \mc{L}$, and thus 
$$
\mc{O}_F(Q_1-Q_2)\cong \det(\mc{E}_{2,0}\otimes \mc{L})\cong \mc{L}^{\otimes 2}.
$$
This is absurd by the equality $\ord \mc{O}_F(Q_1-Q_2)=2$. Therefore, we obtain the conclusion.
\end{proof}


\begin{lem}\label{lem:pull_back2}
Let $E$ be an elliptic curve, and take
a line bundle $\mc{L}\in \Pic^0E$ with $m:=\ord \mc{L}<\infty$.
Then there is an isogeny $\varphi\colon F\to E$  of
degree $m$  such that $\varphi^*\mc{L}\cong \mc{O}_F$.
\end{lem}

\begin{proof}
Define  $\varphi$ to be the dual morphism of the quotient morphism
$$
E\cong \Pic^0E \to F:=\Pic^0E/\Span{\mc{L}}.
$$
Then, $\varphi$ satisfies the desired property.
\end{proof}

\begin{lem}\label{lem:semi_ample}
Let $\varphi\colon F\to E$ be a finite morphism of elliptic curves, and $\mc{E}$ be a rank $2$ normalized 
vector bundle on $E$ with $e=-1$ or $0$. Suppose that  $S':=\PP(\varphi^*\mc{E})$ has an 
elliptic fibration $\pi'$. Then we have the following.
\begin{enumerate}
\item
We have $q^*\omega_S\cong \omega_{S'}$, where $q\colon S'\cong F\times_E S\to S:=\PP(\mc{E})$ is the second projection.
\item
The elliptic ruled surface $S$ also has an elliptic fibration $\pi\colon S\to \PP^1$ fitting into 
the following commutative diagram:
\begin{equation}\label{eqn:quotient_construction}
\xymatrix{ 
F \ar[d]_\varphi  \ar@{}[dr]|{\square} &  S'\ar[l]_{f'} \ar[r]^{\pi'} \ar[d]^q  & \PP ^1 \ar[d]^\psi   \\
E      &  S      \ar[l]^{f}       \ar[r]_\pi                            &  \PP ^1
}
\end{equation}
Here the left square is a fiber product diagram, and $\psi\circ \pi'$ is the 
Stein factorization of $\pi\circ q$. 
\end{enumerate}
\end{lem}

\begin{proof}
(i)
If $\varphi$ is \'{e}tale, the result is obvious. Thus we may assume that $\varphi$ is purely inseparable of degree
 $p$, that is $\varphi=\Fr$. In this case, $\Omega_{F/E}=\Omega_F=\mc{O}_F$, hence $\Omega_{S'/S}
\cong f'^*\Omega_{F/E}=\mc{O}_S'$. Then we obtain the result by \cite[Corollary 3.4]{MR927978}.

(ii) Take a sufficiently large integer $m$. Then we may assume that $h^0(S',\omega_{S'}^{\otimes -m})$ is sufficiently 
large.
Now, we have a series of equalities
\begin{align*}
h^0(S',\omega_{S'}^{\otimes -m})&=h^0(E,f_*q_*q^*\omega_{S}^{\otimes -m})=h^0(E,\varphi_*f'_*q^*\omega_{S}^{\otimes -m})\\
 &=h^0(E,\varphi_*\varphi^*f_*\omega_{S}^{\otimes -m})=h^0(E,f_*\omega_{S}^{\otimes -m}\otimes \varphi_*\mc{O}_F).
\end{align*}
 Here, the third equality comes from the flat base change theorem.

If $\widehat{\varphi}$ is separable, 
it follows from Lemma \ref{lem:push_forward} that there exists an invertible sheaf $\mc{L}\in\Pic^0 E$ such that 
$\varphi^*\mc{L}\cong \mc{O}_F$ and 
$$
h^0(S,\omega_{S}^{\otimes -m}\otimes f^*\mc{L}))=h^0(E,f_*\omega_{S}^{\otimes -m}\otimes \mc{L})\ge 2.
$$
Next, suppose that $\widehat{\varphi}$ is purely inseparable of degree $p$. Since 
$\varphi_*\mc{O}_F\cong \mc{E}_{p,0}$  by Lemma \ref{lem:push_forward} and 
$\mc{E}_{p,0}$ has a filtration into $\mc{O}_E$ by \eqref{eqn:Er0}, we have
$$
h^0(S,\omega_{S}^{\otimes -m})=h^0(E,f_*\omega_{S}^{\otimes -m})\ge 2.
$$

Since, in general,  $\varphi$ is a composition of separable morphisms and purely inseparable morphisms of
 degree $p$, we can find a divisor $H$ on $S$ satisfying $H^2=K_S\cdot H=0$ and $h^0(S,\mc{O}_S(H))\ge 2$.
Let us consider the moving part $|M|$ of the complete linear system 
$
|H|.
$
Note that every effective divisor on $S$ has a non-negative self-intersection, 
hence $M^2=0$ and $h^0(S,\mc{O}_S(M))\ge 2$. Thus we know that the linear system $|M|$
is base point free, and  it satisfies $M\cdot K_S=0$. 
Let us take the Stein factorization of the projective morphism defined by $|M|$, 
and then we obtain a morphism $\pi\colon S\to \PP^1$ with connected fibers.  
\cite[Theorem 2]{MR1969} forces the general fiber of $\pi$ is reduced.
Lemma \ref{lem:fiber_on_ruled} kills the possibility that $\pi$ is quasi-elliptic, hence $\pi$ is 
elliptic. 

Taking the Stein factorization of $\pi\circ q$, we obtain a finite morphism $\psi$ in the diagram.
\end{proof}

We apply Lemma \ref{lem:semi_ample} to show the following.

\begin{prop}\label{prop:semi_ample_e_-1}
Let $\mc{E}$ be a vector bundle of rank $2$ on an elliptic curve $E$. In the following cases, $\PP(\mc{E})$ has an elliptic fibration.
\begin{enumerate}
\item
$p>0$ and $\mc{E}=\mc{E}_{2,0}$. 
\item
$p\ge 0$ and $\mc{E}=\mc{O}_E\oplus \mc{L}$ for $\mc{L}\in \Pic ^0E$
with $\ord \mc{L}<\infty$. 
\item
$p\ge 0$ and $\mc{E}=\mc{E}_Q$. 
\end{enumerate}
\end{prop}

\begin{proof}
(i)
Take an isogeny $\varphi\colon F\to E$ of degree $p$ such that $\widehat{\varphi}$
is purely inseparable, namely $\varphi:=\widehat{\Fr}$. Then we have an isomorphism $\PP(\varphi^*\mc{E}_{2,0})\cong \PP^1\times E$ by Lemma \ref{lem:pull_back} (i). Then, 
the result follows from Lemma \ref{lem:semi_ample}.  

(ii) Lemma \ref{lem:pull_back2} tells us that  $\varphi^*(\mc{O}_E\oplus \mc{L})\cong \mc{O}_F\oplus \mc{O}_F$ for a suitable isogeny $\varphi\colon F\to E$. Hence the result follows from  Lemma \ref{lem:semi_ample}.

(iii) Suppose that $p=2$. Then  Lemma \ref{lem:pull_back} (ii) assures the existence of isomorphism 
$\PP(\Fr^*\mc{E}_{Q})\cong \PP(\mc{E}_{2,0})$. Then, combining (i) with  Lemma \ref{lem:semi_ample}, we obtain the conclusion.

Suppose that $p\ne 2$. Then  Lemma \ref{lem:pull_back} (iii) implies that there is an isomorphism 
$\PP(\varphi^*\mc{E}_{Q})\cong \PP(\mc{O}_F\oplus \mc{L})$, where 
 $\varphi\colon F\to E$ is an isogeny of degree $2$, and $\mc{L}\in \Pic^0F$ with $\ord \mc{L}=2$. 
Combining (ii) with  Lemma \ref{lem:semi_ample}, we obtain the conclusion.  
\end{proof}

\begin{rem}\label{rem:maruyama}
Maruyama showed in \cite[Lemma 9]{MR280493} that $\PP(\mc{E}_Q)$ ($\mathbf{P}_1$ in his notation) has a base point free
 linear pencil whose generic member is an elliptic curve if $p\ne 2$. In addition to it, he also stated in 
 \cite[Remark 7]{MR280493} that a similar result is true in the case $p=2$ ``by reduction''. The authors do not 
 understand what it means.

Using this result, he showed which elliptic ruled surfaces have an elliptic fibration in \cite[Theorem 4]{MR280493}.
But the authors feel that the proof in the case $p=2$ is quite unsatisfactory.
\end{rem}

\subsection{Ramifications and singular fibers}
Let us consider the situation in Lemma \ref{lem:semi_ample} and the diagram \eqref{eqn:quotient_construction}. 
%
If an ordinary elliptic curve is isogeneous to another elliptic curve, it is also ordinary. Therefore,  we see that
all of $E$, $F$, general fibers of $\pi$ and $\pi'$, and reductions of multiple fibers of $\pi$ and $\pi'$ are either ordinary or supersingular elliptic curves simultaneously. 
In this subsection, we extract information on multiple fibers of $\pi$ from one of $\pi'$, 
and vice versa.


\begin{lem}\label{lem:mult_branch}
In the diagram \eqref{eqn:quotient_construction}, fix a point $Q\in \PP^1$ and $\psi^*Q=\sum e_iQ_i'$, where $Q_i'\ne Q_j'$ for $i\ne j$.
Denote the fiber of $\pi$ over $Q$  by $mD$, where $m$ is its multiplicity, and 
denote the fiber $m'_iD'_i$ of $\pi'$ over the point $Q_i'$ similarly.
\begin{enumerate}
\item
For each $i$, $m'_ie_i$ is divisible by $m$. Moreover we have
$$
\sum \frac{m_i'e_i}{m}\le \deg q.
$$
\item
If $m=1$ holds, then we have $m_i'=1$ for all $i$. Conversely if $e_i=m_i'=1$ holds for some $i$, then we have $m=1$.
\item
Suppose that $\psi$ is not ramified at $Q'_i$, that is $e_i=1$. Then 
$\pi$ has a multiple fiber over $Q$ if and only if $\pi'$ has a multiple fiber over $Q'_i$.  
\end{enumerate}
\end{lem}

\begin{proof}
(i) The assertions are a direct consequence of the equalities
\begin{equation}\label{eqn:mDQ}
q^*(mD)=q^*\pi^*Q=\pi'^*\psi^*Q=\sum m_i'e_iD_i'.
\end{equation}

(ii) First assume that $m=1$. Then the inequality 
$$
\sum m_i'e_i\le\deg q=\deg \psi=\sum e_i
$$ 
forces that $m_i'=1$ for all $i$.  
The second assertion is a direct consequence of (i).

(iii) The assertion is a direct consequence of (ii).
\end{proof}

\begin{lem}\label{lem:deg2}
In the notation in Lemma \ref{lem:mult_branch}, we assume furthermore that $\deg q=\deg \psi=2$. 
Let us consider the morphism $q|_{D_i'}\colon D_i'\to D$.
\begin{enumerate}
\item
Suppose that $e_1=1$, equivalently $e_2=1$ holds. 
Then we have $m_1'=m_2'=m$, and multiple fibers $m D$, $m D_1'$ and $m D_2'$ 
are either tame or wild simultaneously.
\item
Suppose that $p=m=2$ and $e_1=2$,  $m_1'=1$, and $q|_{D_1'}$ is separable. Then $2D$ is a wild fiber.   
\item
Suppose that $p=m=2$ and $e_1=2$,  $m_1'=2$, and  $q|_{D_1'}$ is isomorphic. Then $2D$ is a wild fiber.   
\end{enumerate}
\end{lem}

\begin{proof}
Note that $q|_{D_i'}^*\mc{O}_D(D)\cong \mc{O}_{D_i'}(q^*D)$.

(i) Lemma \ref{lem:mult_branch} (i) tells us that $m_1'=m_2'=m$.
 We can see that $q|_{D_i'}$ is isomorphic 
and $q|_{D_i'}^*\mc{O}_D(D)\cong \mc{O}_{D_i'}({D_i'})$.
Hence, $\ord \mc{O}_D(D)=\ord \mc{O}_{D_i'}({D_i'})$, and the second assertion follows.

(ii) Note that $q^*D=D_1'$ in this case.  The morphism $q|_{D_1'}$ is either the quotient morphism 
by $\Z/2\Z$ or  isomorphic, hence 
the kernel $\Ker \widehat{q|_{D_1'}}$ of the dual morphism 
is $\mu_2$ in the former case, and is  trivial in the latter case. 
On the other hand, we have $q|_{D_1'}^*\mc{O}_D(D)\cong \mc{O}_{D_1'}({D_1'})\cong \mc{O}_{D_1'}$, and hence 
$\ord\mc{O}_D(D)=1$. In particular, $2D$ is a wild fiber.

(iii) In this case, we have $q^*D=2D_1'$, and hence 
we see $q|_{D_1'}^*\mc{O}_D(D)\cong \mc{O}_{D_1'}(2D_1')\cong \mc{O}_{D_1'}$. Thus $\ord \mc{O}_D(D)=1$, and  therefore, $2D$ is a wild fiber. 
\end{proof}


\section{Multiple fibers on elliptic ruled surfaces}\label{sec:singular_fiber}
 We study multiple fibers of an elliptic fibration $\pi\colon S\to \PP ^1$ on an elliptic ruled surface $S$
 in this section. We use the notation in \S \ref{subsec:elliptic}.


\begin{lem}\label{lem:ngeq2}
We have
$$
 \sum_{i=1}^\lambda \displaystyle \frac{a_i}{m_i}< 2+d.
$$ 
\end{lem}

\begin{proof}
By Proposition \ref{prop:canonical}, we have
\begin{equation}\label{eqn:PnS}
h^0(S,\omega_S^{\otimes n}) = h^0(S,\pi^* \mc{M} \otimes \mc{O}_S(\sum_{i=1}^\lambda(na_i - m_i[\frac{na_i}{m_i}])D_i)),
\end{equation}
where we set
$$
\mc M:= \omega_{\PP^1}^{\otimes n} \otimes \mc L_{\pi}^{\otimes -n} 
\otimes \str_{\PP^1}(\sum_{i=1}^\lambda \displaystyle [\frac{na_i}{m_i}]).
$$
Then since $\sum_{i=1}^\lambda(na_i - m_i[\frac{na_i}{m_i}])D_i$ 
is a fixed part of the linear system $|nK_S|$,  we have
\begin{align*}
0=h^0(S,\omega_S^{\otimes n}) =  h^0(S,\pi^*\mc M )=  h^0(\PP^1,\mc M ).
\end{align*}
Hence we obtain  
$
\sum_{i=1}^\lambda [\frac{na_i}{m_i}] <n(d+2)
$
for all $n>0$.
This is equivalent to the desired inequality
$$
 \sum_{i=1}^\lambda \displaystyle \frac{a_i}{m_i}< 2+d.
$$ 
\end{proof}

Now we can prove the following lemma.


\begin{lem}\label{lem:kappa-}
The quantities $d, m_i, a_i$ satisfy the following.
\begin{center}
\begin{tabular}{|c|c|c|c|c|}
\hline
\mbox{Case}  & $d$ & $m_i$ & $a_i$ & $p$ \\ \hline \hline
$(I)$ & $0$ & no multiple fibers &  & $p \geq 0$ \\ \hline
$(II)$ & $0$ & $(m, m)$, $m>1$&  & $p \geq 0$ \\ \hline
$(III)$ & $0$ & $(2, 2, 2)$ &  & $p \geq 0$ \\ \hline
$(IV)$  & $-1$ & $({p^\alpha}^*)$, $\alpha>0$ & $a_1 = p^\alpha - 1$ & $p > 0$ \\ \hline
$(V)$  & $-1$ & $({p^\alpha}^*)$, $\alpha>0$ & $a_1 = p^\alpha - 2$ & $p > 0$ \\ \hline
$(VI)$ & $-1$ & $(2, 2^*)$ & $a_2 = 0$ & $p = 2$ \\ \hline
\end{tabular} 
\end{center}   
The symbol ${}^*$ denotes a wild fiber.
\end{lem}
\begin{proof}
We can divide into the following two cases;
(1) $d=0$  and (2) $d=-1$. (Equivalently,  $\sum_{i=1}^\lambda \displaystyle \frac{a_i}{m_i} <2$, and $<1$ respectively by Lemma \ref{lem:ngeq2}).

In the case (1), it follows from Proposition \ref{prop:canonical} (iii)
that $\pi$ has no wild fibers. 
In particular, we have
 $a_i = m_i - 1$ for all $i$ by Proposition \ref{prop:canonical} (ii).
Then $\pi$ has no multiple fibers, \emph{or}
$\pi$ has multiple fibers of type ($m$), 
($m_1, m_2$), ($2, 2, m$) or ($2, 3, m_3$), where $m$ and $m_i$ satisfy that
$$
2\le m, m_1,  m_2, \quad 3 \le m_3 \le 5.
$$
In the latter case, apply Theorem \ref{thm:algebraic} (and Remark \ref{rem:m1m2}), and then 
we see that the multiple fibers are of type ($m, m$) or ($2, 2, 2$).
 
In the case (2), Proposition \ref{prop:canonical} implies that 
$h^0 (\PP^1, \mc T_{\pi}) = 1,$
which means that $\pi$ has a single wild fiber.
Then
the inequality $\sum_{i=1}^\lambda {a_i}/{m_i} <1$ implies that
 the multiple fiber is unique (2-1), \emph{or} that 
$\pi$ has one tame fiber and one wild fiber (2-2). 

(2-1) In this case,  Remark \ref{rem:m1m2} (i) implies that the unique multiple fiber has the multiplicity $p^{\alpha}$ ($\alpha>0$) and $\nu=1$, and  
Theorem \ref{thm:MR491720}  implies  that $a=m-1$ or $m-2$. 

(2-2) Suppose that the type of $\pi$ is $(m_1,m_2^*)$. Then the inequality 
$$
\frac{m_1-1}{m_1}+\frac{a_2}{m_2}<1
$$
and Theorem \ref{thm:MR491720} yield that $a_2=m_2-\nu_2-1$.
Since $m_2$ is of the form $p^\alpha\nu_2$ with $\alpha>0$, 
we can see that $p=2$, the type of $\pi$ is $(2,2^*)$ and
$a_2=0$.
\end{proof}

\section{Proof of Theorem \ref{thm:main1}}\label{sec:main1}
Let $S$ be an elliptic ruled surface admitting a $\PP^1$-bundle structure $f \colon S \rightarrow E$ over an elliptic curve $E$. We use the notation in \S \ref{subsec:elliptic_ruled}.


\begin{lem}\label{lem:keyprop}
Suppose that $S$ has an elliptic fibration $\pi \colon S \rightarrow \PP ^1$.
Then the equality $e=0$ holds in the cases (I), (II) and (V) in Lemma \ref{lem:kappa-}, and 
the equality $e=-1$ holds in the cases (III), (IV) and (VI) in Lemma \ref{lem:kappa-}.
\end{lem}

\begin{proof}
First of all, we have $e=0$ or $-1$, as is mentioned in \S \ref{subsec:atiyah}.  
Note that the multiplicities of all multiple fibers of $\pi$ 
are common by Lemma \ref{lem:kappa-}.
Hence $D_i \equiv D_j$ for each $i,j$ in the notation in \S \ref{subsec:elliptic}..  
We denote the common multiplicities by $m$. If $\pi$ has no multiple fibers, we put $m=1$ for simplicity.

Proposition \ref{prop:canonical} implies that 
$$
K_S\sim (-2-d)F_{\pi}+\sum _ia_iD_i\equiv ((-2-d)m+\sum_ia_i)D,
$$
where $D$ is a reduction of a multiple fiber of $\pi$.
By a direct computation, we have
\begin{equation*}
(-2-d)m+\sum _i a_i=
\begin{cases}
         -2   & \text{in the cases (I),(II) and (V),} \\
         -1 & \text{in the cases (III), (IV) and (VI).}
\end{cases}
\end{equation*}
Consequently in the cases (I), (II) and (V),
the integer $e= K_S \cdot C_0$ should be even, i.e.~$e=0$.
In the cases (III), (IV) and (VI), 
we have
$
D\cdot F_f=-K_S\cdot F_{f}=2.
$
Combining this with the equality $C_0\cdot F_f=1$,
 we conclude that $C_0$ cannot be contained in a fiber of $\pi$, equivalently $e=-(C_0)^2\ne 0$. This means $e=-1$.
\end{proof}

\begin{rem}\label{rem:2DD}
In the proof of Lemma \ref{lem:keyprop}, we show that 
\begin{equation*}
K_S\equiv
\begin{cases}
         -2D\equiv -2C_0   &\mbox{ if } e=0, \\
         -D\equiv -2C_0+F_f & \mbox{ if } e=-1.
\end{cases}
\end{equation*}
In the diagram \eqref{eqn:quotient_construction}, 
take a point $Q\in \PP^1$ and put $Q':=\psi(Q)$. Denote the fiber over the point $Q'$ by $m'D'$, where $m'$ is 
its multiplicity, and the fiber over the point $Q$ by $mD$ similarly. 
Suppose that $e=0$ for $S'$, and $e=-1$ for $S$. Since $q^*K_S=K_{S'}$ holds by Lemma 
\ref{lem:semi_ample} (i), we obtain 
\begin{equation}\label{eqn:2DD}
2D'\equiv q^*D.
\end{equation}
 \end{rem}


We now give the proof of Theorem \ref{thm:main1}.

\noindent \emph{Proof of Theorem \ref{thm:main1}.}
(i) 
In the case $e=0$, we start with the following claim. 

\begin{cla}\label{cla:palpha}
Assume that $S$ has an elliptic fibration $\pi$ admitting at least one multiple fiber. 
Then $\pi$ is of type $(m,m)$
if and only if $\mathcal{E}$ is decomposable.
\end{cla}
\begin{proof}
Recall first that 
$\mc E$ is decomposable if and only if there exist two sections $C_1$ and $C_2$ of $f$ such that $C_1\cap C_2=\emptyset$ (\cite[Exercise V.2.2]{MR0463157}).
Moreover,  Remark \ref{rem:2DD} says that in the case $e=0$, $D\equiv C_0$, where $mD$ is a multiple fiber 
with multiplicity $m$. To the contrary, if some irreducible curve $D$ satisfying $D\equiv C_0$, it turns out that 
$D$ is the reduction of some multiple fiber of $\pi$, since $-K_S\cdot D=0$.

Suppose that $\pi$ has two multiple fibers $mD_1$ and $mD_2$ with multiplicities $m$. 
Then $D_1$ and $D_2$ are sections of $f$, and hence, $\mc{E}$ is decomposable.

Next, suppose that $C_1$ and $C_2$ are two sections of $f$ satisfying $C_1\cap C_2=\emptyset$.
Then we can put $C_1 \equiv C_0 + b_1 F_f$ and $C_2 \equiv C_0 + b_2 F_f$ for some $b_1, b_2 \geq 0$.
The equalities
$$
0 =C_1 \cdot C_2= C_0^2 + (b_1 + b_2)C_0 \cdot F_f = b_1 + b_2
$$
yield $b_1 = b_2 = 0$.
Hence $C_1$ and $C_2$ are the reductions of some multiple fibers of $\pi$. Therefore, 
Lemma \ref{lem:keyprop} says that this situation fits into the case $\pi$ is of type $(m,m)$.
\end{proof}
\noindent

Let us consider the case $\mc E$ is decomposable. Then, we can take $\mc L\in \Pic ^0E$, $\ord \mc L=m>1$ such that $\mathcal{E} = \str_E \oplus \mathcal{L}$.
In this case, since $$h^0(S,\mc{O}_S(nC_0))=h^0(E,\Sym^n( \str_E \oplus \mathcal{L}))$$ holds,
we can see by Lemma \ref{lem:p1}
that the elliptic fibration $\pi$ has a multiple fiber $mC_0$, 
and thus we can apply Claim \ref{cla:palpha} to conclude 
that $\pi$ is of type $(m,m)$.

Next we consider the case $\mc E$ indecomposable, that is,
$\mc E=\mc E_{2,0}$. Then   
combining Lemma \ref{lem:p0} with Lemma \ref{lem:p1}, we know that $S$ has an elliptic fibration $\pi$
in the case $p>0$, and no elliptic fibrations in the case $p=0$.
 We know from Claim \ref{cla:palpha} that 
$\pi$ is of type $(p^{\alpha*})$. Recall that
we have the following diagram:
\[ 
\xymatrix{ 
F \ar[d]_\varphi \ar@{}[dr]|{\square}  &  F\times \PP^1 \ar[l]_{f'} \ar[r]^{\pi'} \ar[d]^q  & \PP ^1   \ar[d]^\psi  \\
E      &  \PP (\mc{E}_{2,0})      \ar[l]_{f}           \ar[r]^{\pi} &  \PP ^1 
}
\]
Here, 
$\widehat{\varphi}\colon E\to F$ is a purely inseparable isogeny of degree $p$.  
Applying Lemma \ref{lem:mult_branch} (i) for $m=p^\alpha$, $m_i'=1$, $e_i\le p=\deg \psi$, we obtain
$\alpha=1$.

(ii) In the case $e=-1$, $S$ has an elliptic fibration by Lemma \ref{prop:semi_ample_e_-1} (iii).
When $p\ne 2$,  we apply Lemma \ref{lem:mult_branch} (i) for the diagram
\begin{equation}\label{eqn:e-1pne2}
\xymatrix{ 
F \ar[d]_\varphi \ar@{}[dr]|{\square}  &  \PP(\mc{O}_F\oplus \mc{L})  \ar[l]_{f'\quad} \ar[r]^{\quad \pi'} \ar[d]^q  & \PP ^1   \ar[d]^\psi  \\
E      &  \PP (\mc{E}_{Q})      \ar[l]_{f}           \ar[r]^{\pi} &  \PP ^1, 
}
\end{equation}
where $\varphi$ is a separable isogeny of degree $2$ and $\mc{L}$ is an order $2$ element in $\Pic^0 F$.
It turns out that the case (IV) is impossible because $\pi'$ is of type $(2,2)$.
Assume that $\pi$ has a wild fiber $2D$. Then $\mu:=\ord \mc{O}_D(D)=1$,  which contradicts with
$p\ne 2$ and \eqref{eqn:mpnu}. Hence, we conclude that the case (VI) is also impossible, and thus 
the case (III) occurs.

When $p=2$ and  $E$ is  supersingular, then Lemma \ref{lem:multiple_sup} implies that the  
cases (III) and (VI) in Lemma \ref{lem:kappa-} do not occur. Hence, 
the unique possibility is the case (IV) by Lemma \ref{lem:keyprop}, that is, $S$ is an elliptic surface of type $(2^\alpha)$. We have the following diagram:
\[ 
\xymatrix{ 
E_2 \ar[d]_\Fr \ar@{}[dr]|{\square}  &  \PP(\mc{E}_{2,0}) \ar[l]_{f'} \ar[r]^{\pi'} \ar[d]^q  & \PP ^1   \ar[d]^\psi  \\
E      &  \PP (\mc{E}_{Q})      \ar[l]_{f}           \ar[r]^{\pi} &  \PP ^1 
}
\]
Since $\PP(\mc{E}_{2,0})$ has the unique multiple fiber of multiplicity $2$, 
it forces that $\alpha=1$ by \eqref{eqn:2DD}.

When $p=2$ and $E$ is ordinary, we want to show the case (VI) occurs.
It suffices to exclude the cases (III) and (IV) by Lemma \ref{lem:keyprop}.  
Let us consider the diagram \eqref{eqn:e-1pne2} again.
To obtain a contradiction, first assume that the case (IV) occurs, that is, 
$\pi$ has the unique multiple fiber $2^\alpha D$ over a point $Q$, and $2^\alpha D$ is a wild fiber.
Since $\pi'$  has no wild fibers, Lemma \ref{lem:deg2} (i)
yields that $\psi$ is branched over $Q$, thus $\psi^*Q=2Q'$ for some $Q'\in \PP^1$.
But in this case, it follows from  Lemma \ref{lem:mult_branch} (ii) that 
$\pi'$ has no multiple fibers over points besides $Q'$, which induces a contradiction 
since $\pi'$ is of type $(2,2)$.

Next, assume that the case (III) occurs, that is, 
$\pi$ has three multiple fibers $2D_i$ over points $Q_i$ ($i=1,2,3$). 
Since $\psi$ is separable, its ramification divisor has degree $2$ by the Hurwitz formula.
Because $\psi$ is wildly ramified, \cite[Proposition IV.2.2(c)]{MR0463157} yields that there is the unique branch point in 
$\PP^1$. Hence, we may assume that $Q_1$ and $Q_2$ are not branch points of $\psi$, and then   
$\PP(\mc{O}_F\oplus \mc{L})$ has at least $4$ multiple fibers by Lemma \ref{lem:mult_branch} (iii). This is absurd. \qed

\begin{rem}
It seems worthwhile to summarize the connection between 
known examples in the literature and the result in Theorem \ref{thm:main1}.
\begin{enumerate}
\item 
When $p>0$ and a vector bundle $\mc E$ is of the form 
$\mc O_E\oplus \mc L$ with $\ord \mc L=p$ on an ordinary elliptic curve $E$, 
 then the elliptic surface in \cite[Example 4.9]{MR799664} is isomorphic to  $\mb P(\mc E)$.
\item
When $p>0, e=0$ and $\mc E$ is an indecomposable vector bundle on an ordinary (resp.~supersingular) 
elliptic curve $E$, 
then the elliptic surface in \cite[Example 4.7]{MR799664} (resp.~\cite[Example 4.8]{MR799664}) 
is isomorphic to $\mb P(\mc E)$. Here we use the fact that an elliptic curve which is isogeny to an 
ordinary (resp.~supersingular) elliptic curve is again ordinary (resp.~supersingular).
\item
Let $E$ be an elliptic curve, and consider the morphism
$$ 
a\colon E\times E\to E \qquad (x_1,x_2)\mapsto x_1+x_2
$$
defined by the addition, and the morphism
$$
i\colon   E\times E\to E \qquad (x_1,x_2)\mapsto x_1-x_2
$$
defined by the subtraction. Then 
we obtain the following commutative diagram:
\begin{align*}
\xymatrix{ 
E  \ar@{=}[d]   &   E\times E    \ar[l]_{a} \ar[r]^{i}\ar[d]^{q'} & E  \ar[d]^{q}  \\
E                       & S   \ar[l]^{f} \ar[r]_{\pi}        &  \PP ^1 
}
\end{align*}
Here we denote the symmetric product $\Sym ^2 (E)$ by $S$, 
and define $q'$ to be the quotient morphism 
by the involution $(x_1,x_2)\mapsto (x_2,x_1)$,
$q$ to be the quotient morphism by the action $x\mapsto -x$.  
 
Then we know that $q$ is ramified at the origin $O$ of $E$, and the points of order $2$ in $E$.
We can also see that $\pi$ is an elliptic fibration, and $f$ is a $\PP^1$-bundle. 
Furthermore the equality
\begin{align*}
&\bigl\{ p\in \PP^1 \mid \pi^*(p) \mbox{ is a multiple fiber}\bigr  \}\\
=&\bigl\{ q(p') \in \PP^1 \mid  \mbox{ $\ord p'=2$ in $E$}\bigr   \}
\end{align*}
holds. We also see that every multiple fiber has multiplicity $2$.

Suppose first that $p\ne 2$.
Then there are exactly $3$ points of order $2$ in $E$, hence
$S$ fits into the case in Theorem \ref{thm:main1} (ii-1).
 
Secondly suppose that $p=2$ and $E$ is an ordinary elliptic curve. Then since there is a unique point of order $2$ in $E$,
$S$ fits into the case in Theorem \ref{thm:main1} (i-5). (We can actually check that the section $C_0$ of $f$ is the reduction of the 
multiple fiber of $\pi$.)

Finally suppose that $p=2$ and $E$ is a supersingular elliptic curve. Then since there are no points of order $2$ in $E$,
$S$ fits into the case in Theorem \ref{thm:main1} (i-1). 

In \cite[page 451]{MR131423}, it is mentioned that $\Sym^n(E)$ is given as a projective bundle $\PP(\mc{E})$ on $E$, where
$\mc{E}$ is a vector bundle in $\mc{M}_E(n,n-1)$. This statement seems incompatible with the above results in the case $p=2$.
\end{enumerate}
\end{rem}


\section{Proof of Theorem \ref{thm:reduction}}\label{sec:reduction}
In this section we prove Theorem \ref{thm:reduction}.
For this purpose, we illustrate the diagram \eqref{eqn:quotient_construction} in 
Lemma \ref{lem:semi_ample} when $S$ has an elliptic fibration with multiple fibers, namely in the cases
(i-2), (i-5), (ii-1), (ii-2) and (ii-3) in Theorem \ref{thm:main1}.
We also study the ramification of $\psi$, if $\psi$ is separable.

\subsection{Case (i-2): $\PP (\mc{O}_E\oplus \mc{L})$, $p\ge 0$,  $m=\ord \mc{L}<\infty$}\label{subsec:O+L}  
Take the quotient morphism 
$$
E\cong \Pic^0E\to F:=\Pic^0E/\Span{\mc{L}}
$$ 
and define $\varphi$ to be its dual. 
Then we obtain the following diagram:  
\[ 
\xymatrix{ 
F \ar[d]_\varphi \ar@{}[dr]|{\square}  &  F\times \PP^1 \ar[l]_{f'} \ar[r]^{\pi'} \ar[d]^q  & \PP ^1   \ar[d]^\psi  \\
E      &  \PP (\mc{O}_E\oplus \mc{L})      \ar[l]_{f\quad }           \ar[r]^{\quad \pi} &  \PP ^1 
}
\]
Theorem \ref{thm:main1} says that $\pi$ is of type $(m,m)$. Suppose that $\pi$ has multiple fibers over points $Q_1,Q_2$.  
When $m$ is not divisible by $p$, then Lemma \ref{lem:mult_branch} (ii) implies  that $\psi$ is branched over $Q_1,Q_2$. 
When $E$ is ordinary and $m=p^n$ ($n>0$), then 
all of $\varphi$, $q$ and $\psi$ are purely inseparable of degree $p^n$.

\subsection{Case (i-5): $\PP (\mc{E}_{2,0})$, $p>0$}\label{E20p>0} 
Take $\varphi \colon F\to E$ such as $\widehat{\Fr}=\varphi$.
 Then we have the following diagram. 
\begin{equation*}\label{eqn:2,0,p}
\xymatrix{ 
F \ar[d]_\varphi\ar@{}[dr]|{\square}   &  F\times \PP^1 \ar[l]_{f'} \ar[r]^{\pi'} \ar[d]^q  & \PP ^1   \ar[d]^\psi  \\
E      &  \PP (\mc{E}_{2,0})      \ar[l]_{f}           \ar[r]^{\pi} &  \PP ^1 
}
\end{equation*}
We see in Theorem \ref{thm:main1} that $\pi$ has the unique (wild) multiple fiber $pD$ over a point $Q_1\in \PP^1$. 
If $E$ is ordinary, $\varphi$, $q$ and $\psi$ are separable of degree $p$, and we can see from Lemma \ref{lem:mult_branch} 
(ii) that $\psi$ is wildly ramified and $Q_1$ is the unique branch point of $\psi$.
If $E$ is supersingular, all of $\varphi$, $q$ and $\psi$ are purely inseparable of degree $p$.

\subsection{Cases (ii-1), (ii-3): $\PP (\mc{E}_{Q})$, except  $p=2$ and $F$ is supersingular}\label{subsec:pne2}    
In this case, there is a separable isogeny $\varphi\colon F\to E$ of degree $2$. Then there is an order $2$ 
element $\mc{L}\in \Pic^0 F$ and the following diagram:
\[ 
\xymatrix{ 
F \ar[d]_\varphi \ar@{}[dr]|{\square}  &  \PP(\mc{O}_F\oplus \mc{L})  \ar[l]_{f'\quad} \ar[r]^{\quad \pi'} \ar[d]^q  & \PP ^1   \ar[d]^\psi  \\
E      &  \PP (\mc{E}_{Q})      \ar[l]_{f}           \ar[r]^{\pi} &  \PP ^1 
}
\]
When $E$ is ordinary and $p=2$, $\pi$ has one wild fiber $2D_1$ over a point $Q_1$ and one tame multiple 
fiber $2D_2$ over a point $Q_2$. Since $\psi$ is wildly ramified at a single point and $\pi'$ is of type $(2,2)$,  Lemma \ref{lem:deg2} (i) (or (ii)) implies that $Q_1$ is the unique branch point of $\psi$.

When $p\ne 2$, $\pi$ has $3$ multiple fibers over points $Q_1,Q_2,Q_3\in \PP^1$.
It follows from Lemma \ref{lem:deg2} (i) that $\psi$ is branched at two points of $Q_i$'s.

\subsection{Cases (ii-2), (ii-3): $\PP (\mc{E}_{Q})$, $p=2$}\label{subsec:p2}  
In this case, we have the following:
\[ 
\xymatrix{ 
E_2 \ar[d]_\Fr \ar@{}[dr]|{\square}  &  \PP(\mc{E}_{2,0}) \ar[l]_{f'} \ar[r]^{\pi'} \ar[d]^q  & \PP ^1   \ar[d]^\Fr  \\
E      &  \PP (\mc{E}_{Q})      \ar[l]_{f}           \ar[r]^{\pi} &  \PP ^1 
}
\]
Recall that $\pi'$ has the unique wild fiber over a point $Q'\in \PP^1$.

When $E$ is supersingular, $\pi$ also has the unique wild fiber over the point $\Fr (Q')$.
When $E$ is ordinary, $\pi$ has one wild fiber $2D_1$ over a point $Q_1$ and one tame multiple 
fiber $2D_2$ over a point $Q_2$. 
By Lemma \ref{lem:deg2} (iii), we can see $\Fr (Q')=Q_1$.

\subsection{Proof of Theorem \ref{thm:reduction}}

\noindent \emph{Proof of Theorem \ref{thm:reduction}.}
Define $q\colon F\times \PP^1\to S$ to be a suitable successive finite \'etale covers and purely inseparable morphisms of degree $p$ in the previous subsections. Then we obtain Theorem \ref{thm:reduction}. \qed

\begin{rem}\label{rem:strange}
A wild fiber $mD$ is said to be \emph{of strange type} if $a=m-1$. 
Katsura and Ueno show  a reduction of  a wild multiple fiber to a tame multiple fiber 
by a finite cover if no wild fibers of strange type appear in the procedure of reduction (see \cite[\S 6 
and p.330]
{MR799664}).  

Recall that if $E$ is supersingular, the elliptic fibration $\pi\colon \PP_E(\mc{E}_Q)\to \PP^1$ has a multiple fiber satisfying 
$
(a/m)=(1/2^*),
$ 
namely $\pi$ has a wild fiber of strange type. In \S \ref{subsec:p2},
we obtain a reduction of a wild multiple fiber of strange type to a tame multiple fiber. 
\end{rem}



\

\noindent
Takato Togashi

TOHO GIRLS' JUNIOR AND SENIOR HIGH SCHOOL,
1-41-1 Wakaba-cho, 
Chofu, Tokyo, 182-8510, Japan

{\em e-mail address}\ : 
\  t.togashi@toho.ac.jp
\ \vspace{0mm} \\

\noindent
Hokuto Uehara

Department of Mathematical Sciences,
Graduate School of Science,
Tokyo Metropolitan University,
1-1 Minamiohsawa,
Hachioji,
Tokyo,
192-0397,
Japan 

{\em e-mail address}\ : \  hokuto@tmu.ac.jp
\ \vspace{0mm} \\

\end{document}